\documentclass[a4paper,12pt,reqno]{amsart}

\usepackage[english]{babel}
\usepackage[latin1]{inputenc}
\usepackage[T1]{fontenc}
\usepackage{amsfonts}
\usepackage{amsmath}
\usepackage{amsthm}
\usepackage{amscd}
\usepackage{float}
\usepackage{enumerate}
\usepackage{multirow}
\usepackage{graphicx,wrapfig,lipsum}
\usepackage{latexsym}
\usepackage{amssymb}
\usepackage[all]{xy}
\usepackage{mathrsfs}
\usepackage{url}
\usepackage[lastpage,user]{zref}
\usepackage{array}
\usepackage{calc}

\theoremstyle{plain}
\newtheorem{teor}{Theorem}[section]
\newtheorem{cor}[teor]{Corollary}
\newtheorem{prop}[teor]{Proposition}
\newtheorem{lemma}[teor]{Lemma}
\newtheorem{defn}[teor]{Definition}

\theoremstyle{plain}

\theoremstyle{definition}
\newtheorem{oss}[teor]{Remark}

\newfloat{tabella}{H!}{tabella}
\floatname{tabella}{Tabella}
\newfloat{immagine}{H}{imm}
\floatname{immagine}{Figura}

\newcommand{\A}{\mathscr{A}}

\newcommand{\B}{\mathscr{B}}

\newcommand{\Csharp}{\mathscr{C}}
\newcommand{\C}{\mathbb{C}}

\newcommand{\E}{\mathscr{E}}
\newcommand{\F}{\mathscr{F}}

\newcommand{\I}{\mathscr{I}}
\newcommand{\sL}{\mathscr{L}}

\newcommand{\sO}{\mathscr{O}}

\renewcommand{\P}{\mathbb{P}}

\newcommand{\T}{\mathscr{T}}
\newcommand{\Z}{\mathbb{Z}}

\newcommand{\sV}{\mathscr{V}}

\newcommand{\href}[2]{#2}

\makeatletter
\newcommand*{\rightleftarrow}[2]{\mathrel{
  \settowidth{\@tempdima}{$\scriptstyle#1$}
  \settowidth{\@tempdimb}{$\scriptstyle#2$}
  \ifdim\@tempdimb>\@tempdima \@tempdima=\@tempdimb\fi
  \mathop{\vcenter{
    \offinterlineskip\ialign{\hbox to\dimexpr\@tempdima+1em{##}\cr
    \rightarrowfill\cr\noalign{\kern.5ex}
    \leftarrowfill\cr}}}\limits^{\!#1}_{\!#2}}}
\makeatother

\newcommand{\pullbackcorner}[1][dr]{\save*!/#1-1.2pc/#1:(-1,1)@^{|-}\restore}

\newcommand{\im}{\operatorname{Im}}

\DeclareMathOperator{\sfhom}{\mathscr{H}\!\mathit{om}}

\newcommand{\st}{\text{ s.t. }}

\newcommand{\coh}{{\operatorname{Coh}}}

\newcommand{\der}{{\textbf{D}}}

\newcommand{\eend}{{\operatorname{End}}}

\newcommand{\nota}[1]{{}}

\newcommand{\sing}{{\operatorname{Sing}}}
\newcommand{\rk}{{\operatorname{rk}}}

\makeatletter
\let\cite\relax
\DeclareRobustCommand{\cite}{%
  \let\new@cite@pre\@gobble
  \@ifnextchar[\new@cite{\@citex[]}}
\def\new@cite[#1]{\@ifnextchar[{\new@citea{#1}}{\@citex[#1]}}
\def\new@citea#1{\def\new@cite@pre{#1}\@citex}
\def\@cite#1#2{[{\new@cite@pre\space#1\if\relax\detokenize{#2}\relax\else, #2\fi}]}
\makeatother

\begin{document}
\title{The derived category of a non generic cubic fourfold containing a plane}
\author{RICCARDO MOSCHETTI}

\address{Department of Mathematics and Natural Sciences, University of Stavanger, NO-4036 Stavanger, Norway}
	\email{riccardo.moschetti@uis.no}


\begin{abstract}
We describe an Azumaya algebra on the resolution of singularities of the double cover of a plane ramified along a nodal sextic associated to a non generic cubic fourfold containing a plane. We show that the derived category of such a resolution, twisted by the Azumaya algebra, is equivalent to the Kuznetsov component in the semiorthogonal decomposition of the derived category of the cubic fourfold.
\end{abstract}
\maketitle

\section{Introduction}

A \textit{cubic fourfold} $Y$ is a smooth hypersurface of degree three in $\P^5$. The derived category of $Y$ has the following semiorthogonal decomposition
\vspace{-1 mm}
\begin{equation*}
\der^b(Y) = \langle \T_Y, \sO_Y, \sO_Y(1), \sO_Y(2) \rangle
\end{equation*}
\noindent
where $\sO_Y(i)$ is an exceptional object for $i=0,1,2$ and $\T_Y$, the Kuznetsov component, is the admissible subcategory of $\der^b(Y)$ right orthogonal to $\langle \sO_Y, \sO_Y(1), \sO_Y(2) \rangle$. Kuznetsov proved that $\T_Y$ has many properties in common with the derived category of a K3 surface, and this is very related with the problem of rationality of cubic fourfolds, as described in \cite{RATPROB}. Hassett, in \cite{HASSETT}, studied the divisors $\Csharp_d$ of the moduli space of all cubic fourfolds that parametrize cubic fourfolds containing a surface non homologous to a complete intersection, also called \textit{special cubic fourfolds}. The integer $d$ is the discriminant of the saturated sublattice of $H^4_{prim}(Y;\Z)$ spanned by the square of the hyperplane class and by the class of the surface which is not homologous to a complete intersection. Huybrechts in \cite{HUYK3} studied deeper the relations between properties of the component $\T_Y$ and the divisors $\Csharp_d$. 

It is known that $\Csharp_8$ parametrizes cubic fourfolds containing a plane. Let $Y$ be a generic cubic fourfold in $\Csharp_8$ and let $A$ be the plane contained in $Y$. Consider the projection map to another plane $B$ in $\P^5$, disjoint from $A$. The preimage along this map of a point $b \in B$, obtained by intersecting $Y$ with a $\P^3$ spanned by $A$ and $b$, consists of the plane $A$ and, since the degree of $Y$ is three, a quadric surface. By blowing up the plane $A$ one obtains a quadric fibration $\pi: Y^+ \rightarrow B$.

The locus in $B$ which parametrizes the singular fibers is a sextic curve $C$, that turns out to be smooth if the cubic fourfold $Y$ is generic in $\Csharp_8$. In the following, $C$ will be called the sextic associated to the pair $(Y,A)$; notice however that a very general $Y$ in $\Csharp_8$ contains a unique plane, by \cite{VOISIN}. Since $Y$ is smooth, Proposition $1.2$ of \cite{BPrym} applied to the quadric fibration $\pi: Y^+ \rightarrow B$ ensures that this sextic exists and has at most ordinary double points. One can also relate a K3 surface $S$ to the cubic fourfold $Y$ by taking the double covering of the plane $B$ ramified along $C$.

From the homological point of view, Kuznetsov proved the following

\begin{teor}[{\cite[Theorem 4.3]{KUZNETSOV}}] \label{thm:kuznets}
Let $Y$ be a generic cubic fourfold in $\Csharp_8$. Then there exists an exact equivalence of triangulated categories $\T_Y \cong \der^b(S,\tilde{\B_0})$.
\end{teor}
\noindent
Here, $\der^b(S,\tilde{\B_0})$ stands for the derived category of coherent sheaves on $S$ twisted by $\tilde{\B_0}$, a sheaf of Azumaya algebras with the property that the pushforward under the map of the double covering is isomorphic to the sheaf of even Clifford algebras $\B_0$ associated with $\pi$. 
It is interesting to see how it is possible to attach a Brauer class on a smooth K3 surface to $\B_0$ when the sextic associated to the cubic fourfold and the plane is nodal. It means that some fibres of $\pi$ are the union of two different planes and the related quadratic form is no longer simply degenerate. These fibres correspond exactly to the nodes of the sextic. In this case the double covering $S$ is singular, and one needs to resolve the singularities in order to obtain a K3 surface $S^+$. The Clifford algebra $\B_0$ can still be seen as the pushforward of a certain Azumaya algebra $\tilde{\B_0}$ on $S$ but only on the complement of the singular locus of the sextic, and this is not sufficient to prove Theorem \ref{thm:kuznets}.

It is particularly intriguing to look for a way to make the usual theory work in this singular case as well. More precisely, is it possible to define an Azumaya algebra $\A$ on the K3 surface $S^+$ such that the push forward of $\A$ to the plane $B$ is isomorphic to the Clifford algebra $\B_0$? Proposition \ref{prop:final} answers this question positively and allows us to prove the following a generalization of Theorem \ref{thm:kuznets}. 

\begin{teor} \label{teor:extension}
Let $Y$ be a cubic fourfold containing a plane such that the associated sextic curve $C$ on the plane $B$ is nodal. Then there exists an exact equivalence of triangulated categories  $\T_Y \cong \der^b(S^+,\A)$, where $S^+$ is the K3 surface obtained by resolving the singularities of the double covering of $B$ ramified along $C$ and $\A$ is an Azumaya algebra on $S^+$ such that the pushforward to the plane $B$ is isomorphic to the Clifford algebra $\B_0$.
\end{teor}

A geometrical meaning to the Kuznetsov component $\T_Y$ has already been established in \cite{T} for cubic forufolds containing a plane and in \cite{BLMS} for all cubic fourfolds.

\textbf{The plan of the paper.} Some preliminaries concerning quadric fibrations and Azumaya algebras are given in Section \ref{Sec:prel}. The geometric context in which the construction takes place is described in Section \ref{Sec:geometriccontext}. Section \ref{Sec:nodalsextic} is devoted to finding the Azumaya algebra $\A$, which is done in Proposition \ref{prop:final} and to giving the proof of Theorem \ref{teor:extension}. In all the paper we will work over the field of the complex numbers.

\section{Preliminaries} \label{Sec:prel}
\subsection{Quadric fibrations and line bundle valued quadratic forms}
Here is a collection of results concerning quadric fibrations of low dimension. A good tool to deal with such objects is provided by quadratic forms with values in a line bundle. One can associate to a line bundle valued quadratic form a particular quadric fibration in such a way that some geometric properties of the fibration reflect on the quadratic form and vice versa. 

\begin{defn} \label{Defn:LBVQuadraticForm}
Let $B$ be a scheme. A line bundle valued quadratic form on $B$, often simply called a quadratic form, is a triple $(\E,q,\sL)$ where $\E$ is a vector bundle on $B$, $\sL$ is a line bundle on $B$ and $q: \E \rightarrow \sL$ is a morphism of sheaves such that
\begin{itemize}
	\item $q(av)=a^2q(v)$ where $a$ is a section of $\sO_B$ and $v$ is a section of $\E$
	\item the morphism $b_q:\E \times \E \rightarrow \sL$, defined for every $v$ and $w$ sections of $\E$ by 
	$$b_q(v,w)=q(v+w)-q(v)-q(w)$$
	 is $\sO_B$-bilinear.
\end{itemize}
The dimension of the quadratic form $(\E,q,\sL)$ is the rank of $\E$. A quadratic form $(\E,q,\sL)$ is called regular if the morphism from $\E$ to $\sfhom(\E,\sL)$ induced by $b_q$ is an isomorphism. Let $B$ now be a noetherian separated integral scheme; a quadratic form $(\E,q,\sL)$ is called generically regular if the form is regular over the generic point of $B$.
\end{defn}

We will call quadric fibration a morphism $p:X \rightarrow S$ such that the fiber $X_s$ over a point $s$ is a quadric. Starting with a quadratic form $(\E,q,\sL)$, one can consider the projection $\pi:\P_S(\E) \rightarrow S$, where $\P_S(\E)$ is the projectivization of $\E$ on $S$.  By denoting $X \subset \P_S(E)$ as the zero locus of $\sigma$, it is straightforward to prove that the restriction of the projection $\pi$ to $X$ is a quadric fibration. In particular, $X$ is called the quadric fibration associated to $(\E,q,\sL)$. From now on we require all quadric fibrations to be flat, in the same spirit we will use only quadratic form with a flat associated quadric fibration.  Auel, Bernardara and Bolognesi in \cite[Lemma 1.1.1]{ABB} proved the equivalence between Definition \ref{Defn:LBVQuadraticForm} and other definitions of line bundle valued quadratic forms. For instance, to give a quadratic form $(\E,q,\sL)$ is equivalent to give a global section of $\sfhom(\sL^*,S^2 \E^*)$ on $B$, which corresponds to a morphism of $\sO_B$ modules $\sigma:\sL^* \rightarrow S^2 \E^*$. This is useful in order to state the following

\begin{defn}[{\cite[Section 3.5]{QUADRIC}}] \label{defn:degenerationlocus}
Let $(\E,q,\sL)$ be a generically regular quadratic form of dimension $n$ on a scheme $B$. The $d$-th degeneration locus of the quadratic form, denoted by $B_d \subset B$, is a closed subscheme defined by the following sheaf of ideals
$$\I_d=\im(\Lambda^{n+1-d}\E\otimes\Lambda^{n+1-d}\E\otimes (\sL^*)^{n+1-d} \xrightarrow{\Lambda^{n+1-d}\sigma} \sO_B).$$
\end{defn}

\noindent Notice that $B_{i+1}$ is contained in $B_i$ for all positive integers $i$, and that $B_1$ is a divisor on $B$ that is called the discriminant divisor. Geometrically, the $B_1$ parametrizes the singular fibres of the quadric fibration $(\E,q,\sL)$.
The notion of simple degeneration will be useful to describe the situation arising from a generic cubic fourfold containing a plane. The exact definition in terms of quadratic forms can be found in \cite[Section 1.1]{ABB}, whereas a property of simply degenerate quadratic forms will be taken as a definition in this paper: a quadratic form is simply degenerate if and only if its second degeneracy locus $B_2$ is empty. In general $\sing(B_i) \supseteq B_{i+1}$, so $B_2$ being empty does not imply $B_1$ to be smooth, as can be seen for instance in \cite{APS} Proposition 1.5.

\subsection{Clifford algebra and Azumaya algebra}
A good introduction to Clifford algebras is provided by \cite{KNUS}, where Clifford algebras are defined in the context of modules over commutative rings. In the same spirit, one can associate a sheaf of $\Z$-graded algebras to each line bundle valued quadratic form by following a construction proposed in \cite{BK}. The subalgebra of degree zero, denoted by $\B_0$, has a structure of a sheaf of algebras, and is called the even (part of the) Clifford algebra, see \cite{A}, Section 1.8. These spaces play an important geometric role, described for example in \cite{KUZNETSOV}. The construction proposed in \cite{ABB} makes possible to define directly the even Clifford algebra. Starting from a quadratic form on a scheme $B$, not necessarily regular, consider the tensor algebra $T(\E \otimes \E \otimes \sL^*)$, and define the ideals
\begin{align*}
J_1 &= (v \otimes v \otimes f - f(q(v)))\\
J_2 &=(u \otimes v \otimes f \otimes v \otimes w \otimes g - f(q(v))u\otimes w \otimes g)
\end{align*} 
where $u,v,w$ are sections of $\E$ and $f,g$ are sections of $\sL^*$. The even Clifford algebra of the quadratic form $(\E,q,\sL)$ is defined by the quotient
$$\B_0(\E,q,\sL) := T(\E \otimes \E \otimes \sL^*)/(J_1 + J_2)$$

This is not the only way to construct the even Clifford algebra of a quadratic form; in \cite{QUADRIC}, Kuznetsov defines the sheaf of Clifford algebras, and the description of $\B_0$ and $\B_1$ in terms of the initial quadratic form turns out to be
$$\B_0 \cong \sO_B \oplus (\wedge^2 \E \otimes \sL^*) \oplus (\wedge^4 \E \otimes (\sL^*)^2) \oplus \cdots$$
$$\B_1 \cong \E \oplus (\wedge^3 \E \otimes \sL^*) \oplus (\wedge^5 \E \otimes (\sL^*)^2) \oplus \cdots$$
as $\sO_B$-modules.
Other ways to define the even Clifford algebra are the splitting construction and the gluing construction; see \cite[Appendix A]{ABB}, for further references. These ways turn out to be all equivalent, since the Clifford algebra of a quadratic form over a ring is unique. The splitting construction can be used to describe the even Clifford algebra $\B_0$ as a functor. 

\begin{defn} \label{Defn:AzumayaAlgebra}
An algebra $\A$ over a local commutative ring $R$ is said to be Azumaya if $\A$ is a free $R$-algebra of finite rank and such that the map 
\begin{eqnarray*}
\A \otimes_R \A^{op} &\rightarrow& \eend_R(\A)\\
a \otimes b &\mapsto& (x \mapsto axb)
\end{eqnarray*}
is an isomorphism.
If $X$ is a scheme, an $\sO_X$-algebra $\A$ is Azumaya if it is coherent as an $\sO_X$-module and, for every closed point $x$ of $X$, $\A_x$ is Azumaya over $\sO_{X,x}$.
\end{defn}

The even Clifford algebra of a generically regular quadratic form with even dimension can be described in terms of an Azumaya algebra. 

\begin{prop}[{\cite[Proposition 3.13]{QUADRIC}}] \label{Prop:AzumayaAlgebraEvenFibration}
Let $(\E,q,\sL)$ be a generically regular quadratic form on $B$ of even dimension, and let $f:S \rightarrow B$ be the double cover $S$ of $B$ ramified at the discriminant locus $B_1$, determined by the centre of the Clifford algebra $\B_0$. Then there exists a sheaf of algebras $\tilde{\B}_0$ on $S$ such that $f_*(\tilde{\B}_0) \cong \B_0:=\B_0(\E,q,\sL)$ and the functor
$$f_*:\coh(S,\tilde{\B}_0) \xrightarrow{\sim} \coh(B,\B_0)$$ 
is an equivalence of categories. Moreover, the restriction of $\tilde{\B}_0$ to the complement of $f^{-1}(B_2) \subset S$ is a sheaf of Azumaya algebras on $S$.
\end{prop}

Some recent results involving cubic fourfolds containing a plane that make use of the tool of quadric fibrations can be found in \cite{ABBV}, \cite{AFIB}, \cite{G}, \cite{LMS}, \cite{MS}, \cite{GO} and \cite{T}. 

\subsection{Noncommutative varieties and twisted derived categories}
Let $X$ be an algebraic variety and consider a sheaf of $\sO_X$-algebras $\B$ of finite rank as a $\sO_X$-module. The pair $(X,\B)$ is called a noncommutative variety.
A morphism $f:=(f_0, f_{\operatorname{alg}})$ between $(X_1,\B_1)$ and $(X_2,\B_2)$ is given by a morphism of algebraic varieties $f_0:X_1 \rightarrow X_2$ and a morphism of $\sO_{X_1}$-algebras $f_{\operatorname{alg}}:f_0^*\B_2 \rightarrow \B_1$. The morphism $f$ is called strict if $\B_1 \cong f_0^*\B_2$, and is called an extension if $X_1$ is isomorphic to $X_2$, and $f_0$ is the identity. Since the pullback of an Azumaya algebra is still Azumaya, see \cite{KNUS} III.5.1, one can define also Azumaya varieties as noncommutative varieties $(X,\B)$ in which the algebra $\B$ is Azumaya. 
Given a noncommutative variety $(X,\B)$, the so-called twisted derived category $\der^b(X,\B)$ can be defined in the usual way starting with the category of coherent sheaves of right $\B$-modules on $X$ as objects. Some background on noncommutative varieties can be found in \cite{QUADRIC}, and an introduction to Azumaya varieties together with all the definitions of functors involving twisted derived categories can be found in \cite{HYP}.

\begin{prop} \label{prop:pullbackazumvar}
Let $(X,\B_X) \xrightarrow{f} (S,\B_{S})$, and $(Y,\B_Y) \xrightarrow{g} (S,\B_{S})$ be morphisms of noncommutative varieties. Consider the following base change diagram
\begin{equation} \label{dgr:basechange}
	\xymatrix{
 X \times_S Y\ar[d]^{q_0}\ar[r]^-{p_0}\pullbackcorner & X\ar[d]^{f_0} \\
 Y\ar[r]^{g_0} & S
}
\end{equation}
If $f$ is strict, then $(X \times_S Y,q_0^*(\B_Y))$ is a fibre product of $(X,\B_X)$ and $(Y,\B_Y)$ over $(S,\B_{S})$.
If $g$ is strict, then $(X \times_S Y,p_0^*(\B_X))$ is a fibre product of $(X,\B_X)$ and $(Y,\B_Y)$ over $(S,\B_{S})$.
Here the fibre product of noncommutative varieties is defined as in \cite{HYP}, Lemma 10.37.
\end{prop}
\begin{proof}
The proof follows the same lines of Lemma 10.37 in \cite{HYP}, which does not actually require the hypothesis of the algebras being Azumaya.
Assume $f$ to be strict, that is $\B_X = f_0^*\B_S$. First define the morphisms $q:=(q_0,q_{\operatorname{alg}})$ and $p:=(p_0,p_{\operatorname{alg}})$ that makes the following diagram commute
\begin{equation}  \label{dgr:basechange2}
	\xymatrix{
 (X \times_S Y,q_0^*\B_Y)\ar[d]^{q}\ar[r]^-{p} & (X,\B_X)\ar[d]^{f} \\
 (Y,\B_Y)\ar[r]^{g} & (S,\B_S)
}
\end{equation}
The map $q_{\operatorname{alg}}:q_0^* \B_Y \rightarrow q_0^* \B_Y$ is the identity. It remains to define $p_{\operatorname{alg}}:p_0^* \B_X \rightarrow q_0^* \B_Y$. Since $f$ is strict $p_0^* \B_X$ is equal to $p_0^* f_0^* \B_S$, and since Square (\ref{dgr:basechange}) is a base change, $p_0^* f_0^* \B_S$ is equal to $q_0^* g_0^* \B_S$. One can then choose as $p_{\operatorname{alg}}$ the morphism induced by $g_{\operatorname{alg}}$. This choice makes Square (\ref{dgr:basechange2}) commute.
Let now $u:(Z,\B_Z) \rightarrow (X,\B_X)$ and $v:(Z,\B_Z) \rightarrow (Y,\B_Y)$ be morphisms of noncommutative varieties such that $f \circ u = g \circ v$. 
At the level of varieties, the base change (\ref{dgr:basechange}) gives the existence of a morphism $t_0 : Z \rightarrow X \times_S Y$ such that $p_0 \circ t_0 = u_0$ and $q_0 \circ t_0 = v_0$. The morphism $t_{\operatorname{alg}}: t_0^* q_0^* \B_Y \rightarrow \B_Z$ will be the one induced by $v_{\operatorname{alg}}:v_0^*\B_Y \rightarrow \B_Z$. The two properties $v=q \circ t$ and $u=q \circ p$ follows by the definitions of $q$ and $p$.
The same proof holds in the case $g$ is strict.
\end{proof}

The following definition is the analogous for our case of \cite[Definition 2.18]{HYP}.

\begin{defn}  \label{Defn:ExactCartesian}
Let $(X,\B_X) \xrightarrow{f} (S,\B_{S})$, and $(Y,\B_Y) \xrightarrow{g} (S,\B_{S})$ be morphisms of noncommutative varieties. If either $f$ or $g$ is strict, then Proposition \ref{prop:pullbackazumvar} ensures that the fibre product $(X \times_S Y, \B_{X \times_S Y})$ is defined. The square
\begin{equation} \label{Sq:EC}
	\xymatrix{
X \times_S Y \ar[r]^q\ar[d]^p & Y\ar[d]^g\\
X\ar[r]^f & S
}
\end{equation} is said to be exact cartesian if the natural morphism of functors 
$$g^*f_* \to q_*p^*:\der^-(X,\B_X) \to \der^-(Y,\B_Y)$$
is an isomorphism.
\end{defn}

The following lemma follows the same lines of \cite[Corollary 2.27]{HYP}, stated in the context of Azumaya varieties. The difference here is that $\B_0$ and $\bar{\B_0}$ are only Clifford algebras and not Azumaya. For the convenience of the reader we are rewriting the proof.

\begin{lemma} \label{Lemma:ExactCartesian}
In the same notation of Definition \ref{Defn:ExactCartesian}, assume $g$ to be a strict closed embedding, $Y$ a locally complete intersection in $S$ and both $S$ and $X$ to be cohen-Macaulay. If the codimension on $X$ of $X \times_S Y$ is equal to the codimension on $S$ of $Y$, then Square (\ref{Sq:EC}) is exact cartesian.
\end{lemma}
\begin{proof}
Let us first remark the content of \cite[Lemma 2.26]{HYP}. Since the map $g$ is finite, Square (\ref{Sq:EC}) is exact cartesian if the morphism 
\begin{equation} \label{Eqn:morphlemmaEC}
f^*g_* \B_Y \to p_*q^* \B_Y=p_* \B_{X \times_S Y}
\end{equation}
is an isomorphism. For any $F \in \der^-(X,\B_X)$ we get
$$g_*g^*f_*(F) \cong f_*(F) \otimes_{\B_S} g_*\B_Y \cong f_*(F \otimes_{\B_X} p_* \B_{X \times_S Y}) \cong f_* p_* p^* (F) \cong g_*q_*p^*(F).$$
This isomorphism is induced by the pushforward $g_*$ applied to the morphism $g^*f_* \to q_*p^*$, which is a morphism of kernel functors: 
$$f_* \cong \Phi_{(\Gamma_{f})_*\B_X} \qquad g^* \cong \Phi_{(\Gamma_g)_*\B_Y} \qquad q_*p^* \cong \Phi_{\B_{X \times_S Y}},$$
where $\Gamma_f: X \to X \times S$ and $\Gamma_g: Y \to S \times Y$ are the graphs of $f$ and $g$ respectively. We can then apply the result of \cite[Lemma 2.8,b]{HYP}, which only involves the varieties and not the sheaf of algebras, to get the desired isomorphism $g^*f_* \cong q_*p^*$.

We are left to check that the homomorphism (\ref{Eqn:morphlemmaEC}) is an isomorphism. Since the claim is local in $S$, we may assume $Y$ to be the zero locus of a regular section $s$ of a vector bundle $\sV$ on $S$ having rank equal to the codimension of $Y$ in $S$. Recall that $\operatorname{Kosz}_S(s)$ is the Koszul complex of the section $s$:
$$\operatorname{Kosz}_S(s):=\{0 \to \Lambda^{\operatorname{top}} \sV^* \otimes_{\sO_X} \B_X \xrightarrow{s} \ldots  \xrightarrow{s} \sV^* \otimes_{\sO_X} \B_X \xrightarrow{s} \B_X \to 0\}$$
Since $s$ is regular and $X$ is Cohen-Macaulay we have $g_* \B_Y \cong \operatorname{Kosz}_S(s)$. 

The zero locus of $f^*s$ on $X$ is the fibre product $X \times S Y$, but by the hypothesis on the codimension on $X$ of $X \times_S Y$, we have $\operatorname{codim}_X(X \otimes_S Y) = \rk f^* \sV$. Therefore the section $f^*s$ is regular and we get an isomorphism and so $f^*g_*\B_Y \cong \operatorname{Kosz}_X(f^*s)$. This leads to the isomorphism $\operatorname{Kosz}_X(f^*s) \cong p_* \B_{X \times_S Y}$ which concludes the proof.
\end{proof}

\section{The geometric context} \label{Sec:geometriccontext}
Let $Y \subset \P(V)$ be a hypersurface of degree $3$ containing a projective space $A$. Let $V$ be a vector space of dimension $n+1$ and $A$ be $\P(A_0)$ for a vector subspace $A_0$ of $V$. Let $B_0$ be the quotient space $V / A_0$, and $B:=\P(B_0)$. As described in the introduction, it is possible to obtain a quadric fibration $Y^+\rightarrow B$ that is associated to a line bundle valued quadratic form $(\E,q,\sL)$. By using the same notation of \cite{KUZNETSOV}, the vector bundle $\E$ is $A_0 \otimes \sO_B \oplus \sO_B(-1)$, the line bundle $\sL$ is $\sO_B(1)$ and the map $q$ depends on the equations of $Y$. 
\begin{oss}
In the case of a generic cubic fourfold $Y$ in $\Csharp_8$, this quadratic form turns out to be simply degenerate, and then, since $Y$ is also smooth, the vanishing of $B_2$ implies that the sextic curve $C$ on $B$, that coincides with the discriminant divisor $B_1$ of the quadratic form, is also smooth. Hence, one obtains a sheaf of Azumaya algebras on the double cover of the plane $B$ ramified along $C$. Proposition \ref{Prop:AzumayaAlgebraEvenFibration} can also be applied when, for a non generic $Y$, the singular locus $B_2$ of the sextic $C$ turns out to be non empty. The sheaf of algebras obtained is Azumaya only on the complement of $f^{-1}(B_2) \subset S$.
\end{oss}
One can compute the symmetric product of $\E^*$ tensored by $\sL$, according to one of the equivalent definitions of a line bundle valued quadratic form. 
\begin{align*}
S^2\E^* \otimes \sL^* &= S^2\left(A_0^* \otimes \sO_B \oplus \sO_B(1)\right) \otimes \sO(1)\\
                     &= S^2\left(A_0^*\right) \otimes \sO(1) \oplus \left(A_0^* \otimes \sO(2)\right) \oplus \sO(3)
\end{align*}
A section of this bundle is given by a section of $S^2 A_0^*\otimes \sO(1)$, a section of $A_0^* \otimes \sO(2)$ and a section of $\sO(3)$. The same setting can be described in coordinates as follows. Let the ambient space be $\P^n=\P^n(x_0: \cdots : x_b : y_0 : \cdots : y_a)$, with $a+b+1=n$. The space $A$ of dimension $a$ will be $\P^a(y_0: \cdots : y_a)=\{x_0= \cdots = x_b = 0\}$ and the space $B$ will be $\P^b(x_0: \cdots : x_b)=\{y_0= \cdots = y_a = 0\}$. Notice that $A \cap B = \emptyset$. 
Let $\pi$ be the projection map $\P^n \dashrightarrow B$ from the space $A$. One can assume $Y:= \{F=0\}$, 
\begin{equation*}
F:= \sum_{i,j=0}^a l_{ij}y_iy_j + 2\sum_{k=0}^a q_k y_k + f
\end{equation*}
where $l_{ij}$, $q_k$ and $f$ are polynomials in $x_0, \cdots x_b$ of degree one, two and three, respectively, and $l_{ij}=l_{ji}$. One can arrange these polynomials in the following symmetric matrix
$$M:=\left( \begin{array}{ccc|c}
l_{00} & \cdots & l_{0a} & q_0\\
\vdots & \ddots & \vdots & \vdots\\
l_{a0} & \cdots & l_{aa} & q_a\\
\hline
q_0 & \cdots & q_a & f\\
\end{array}\right)$$ 

The discriminant divisor of the quadric fibration $Y^+ \to B$ coincides with the determinant of $M$, which is a hypersurface of degree $a+4$ in $B$.

The matrix $M$ can be seen as an injective map between vector bundles, part of the following exact sequence
\begin{equation} \label{Eqn:MapM}
0 \rightarrow \sO_{\P^b}(-2)^{a+1} \oplus \sO_{\P^b}(-3) \xrightarrow{M} \sO_{\P^b}(-1)^{a+1} \oplus \sO_{\P^b} \rightarrow \F \rightarrow 0
\end{equation}
The following proposition is just a recap of \cite{BDet} and relates the matrix $M$ with a line bundle supported on the hypersurface $\{\det M = 0\}$.
\begin{prop}
The cokernel $\F$ of the map induced by $M$ in (\ref{Eqn:MapM}) is an ACM sheaf supported on the hypersurface $C = \{\det M = 0\}$. If $C$ is smooth, $\F$ is a line bundle. 
\end{prop}
\begin{proof}
This is just a special case of Theorem A and Corollary $1.8$ in \cite{BDet}.
\end{proof}

In the case of a cubic fourfold $Y$ containing a plane, Beauville and Voisin in \cite{BDet}, \cite{BPrym} and \cite{VOISIN} proved that there is a correspondence between $Y$ and the curve $C$ of degree six defined by $\{\det(M)=0\}$ in the plane $B$, together with a theta-characteristic on $C$. The singular case was studied by Stellari in \cite{STHETA}. Notice that, in coordinates, this corresponds to the case $n=5$, $a=b=2$. 

Going back to the derived categories, viewing the cubic fourfold as a quadric fibration is the key for the following proposition, which relates the Kuznetsov component of $\der^b(Y)$ with a twisted derived category. Notice that since $Y$ is smooth, containing a linear subspace half-dimensional, then the arising quadric fibration will be flat. One can perform the same construction carried out in the introduction obtaining a flat quadric fibration with base a plane $B$ disjoint from $A$; we recall the following

\begin{prop}{\cite[Theorem 4.3]{KUZNETSOV}} \label{prop:stepone}
Let $Y$ be a cubic fourfold containing a plane $A$. There is an equivalence
$$\T_Y \cong \der^b(B,\B_0)$$
where $\T_Y$ is the Kuznetsov component of $\der^b(Y)$ and $\der^b(B,\B_0)$ is the derived category of $B$ twisted by the Clifford algebra $\B_0$ related to the quadric fibration.
\end{prop}

\section{Cubic fourfolds in $\Csharp_8$ with an associated nodal sextic} \label{Sec:nodalsextic}
Let's now consider a cubic fourfold $Y$ with an associated sextic $C$ that is nodal. This hypothesis makes $Y$ non generic in $\Csharp_8$. Since $Y$ is smooth, the fibres over the nodes of the sextic are double planes, see for instance \cite{BPrym}. This is equivalent to fact that $B_2=\sing(B_1)$ in terms of the degerenacy locus of the associated line bundle valued quadratic form.

\begin{oss} \label{Rmk:FormSingular}
One can encode the information about the singularity of the sextic in the equation of the cubic fourfold. Up to a change of coordinates, one can assume one node of the sextic to be the point $x:=(1:0:0) \in \P^2(x_0:x_1:x_2)=:B$. The fibre over $x$ is the union of the plane $A$ and two other planes in the space generated by $A$ and $x$, that is $\P^3(x_0:y_0:y_1:y_2)$. One can also assume that, in this $\P^3$, the other two planes have equations $y_1=0$ and $y_2=0$. Hence, the equation of the cubic fourfold $Y$ is $x_0 y_1 y_2 + x_1 k_1 + x_2 k_2 =0$ where $k_1$ and $k_2$ are quadratic polynomials in $x_i$ and $y_i$.
\end{oss}

This geometric picture can be thought as a hyperplane section of a similar picture in $\P^6(x_0:x_1:x_2:x_3:y_0:y_1:y_2)$. Let $\bar{Y}$ be the cubic fivefold defined by the equation $F + x_3 \bar{F}=0$, where $\bar{F}$ is a homogeneous polynomial of degree $2$. $\bar{Y}$ still contains the plane $A$, which in $\P^6$ has equations $\{x_0 = x_1 = x_2 = x_3 = 0\}$ and one can project onto the $\P^3$ defined by $\bar{B} := \{y_0=y_1=y_2=0\}$. Exploiting the construction of $M$ of the previous section, $\bar{Y}$ is given by the equation
\begin{equation*}
\sum_{i,j=0,1,2} l_{ij}y_iy_j + 2\sum_{k=0,1,2} q_k y_k + f + x_3(\sum_{i,j=0,1,2} \bar{l}_{ij}y_iy_j + 2\sum_{k=0,1,2} \bar{q}_k y_k + \bar{f})=0
\end{equation*}
where $\bar{l}_{ij}$, the $\bar{q}_i$ and $\bar{f}$ are polynomials in $x_0$, $x_1$, $x_2$, $x_3$ of degree zero, one and two, respectively, and $\bar{l}_{ij}=\bar{l}_{ji}$.
One can define a symmetric matrix $\bar{M}$, as before:
\begin{equation} \label{Eqn:matrixform}
\left( \begin{array}{c|c}
l_{ij} & q_k\\
\hline
q_k & f\\
\end{array}\right)
+x_3 \left( \begin{array}{c|c}
\bar{l}_{ij} & \bar{q}_k\\
\hline
\bar{q}_k & \bar{f}\\
\end{array}\right)
\end{equation}

\noindent Taking the determinant of this matrix gives rise to a sextic surface $\bar{C}$ in $\P^3(x_0:x_1:x_2:x_3)$.

The idea is to exploit the work of Kuznetsov \cite{K2} to obtain an Azumaya algebra related to the Clifford algebra $\B_0$. The first step consists in proving that is it always possible to find a cubic fivefold as described above with the singularities of the sextic surface $\bar{C}$ being at most isolated nodes. The situation can be reformulated in terms of degeneracy locus of maps between vector bundles. Recall the following
\begin{defn}\label{Defn:DegeneracyLocus}
Let $\E$ and $\F$ be vector bundles on a projective variety $X$, $\phi: \E \rightarrow \F$ a morphism and $k$ a positive integer. Then,
$$D_k(\phi)=\{x \in X \st \rk(\phi_x) \leq k\}$$
is called the $k$-degeneracy locus of $\phi$.
\end{defn}
As in Section \ref{Sec:geometriccontext}, but now for the case of $B = \P^3$, let $\E$ be the vector bundle $3\sO_{\P^3} \oplus \sO_{\P^3}(-1)$ on $\P^3$. As seen in the previous section, a matrix of homogeneous forms on $\P^3$ with the same degrees as the matrix in Equation (\ref{Eqn:matrixform}) defines a symmetric map $\phi: \E \rightarrow \E^\star(1)$ as in Equation (\ref{Eqn:MapM}), after tensoring the exact sequence with the line bundle $\sO_{\P^3}(2)$. This map can be seen as a section of the bundle $S^2 \E \otimes \sO_{\P^3}(1)$ and its symmetrical degeneracy locus $D_3(\phi)$ coincides with the locus of $\P^3$ where the determinant of the matrix (\ref{Eqn:matrixform}) is equal to zero, that is the sextic surface $\bar{C}$.
The choice of the second matrix in (\ref{Eqn:matrixform}) defines a linear system $\T$ of sections of $S^2 \E \otimes \sO_{\P^3}(1)$.
\begin{oss}
Let $\E=3\sO_{\P^3} \oplus \sO_{\P^3}(-1)$ as above. The degeneracy locus $D_3(\phi)$, of Definition \ref{Defn:DegeneracyLocus} coincides with the first degeneration locus of the line bundle valued quadratic form of Definition \ref{defn:degenerationlocus}, $B_{1}(\sigma)$, where $\sigma: \sO_{\P^3}(-1) \rightarrow S^2 \E$ is a map equivalent to $\phi$.
\end{oss}
The next lemma, following \cite{OTT}, holds in general when $X$ is a projective variety.
\begin{lemma} \label{Lemma:SymmetricCodimension}
Let $\E$ be a vector bundle of rank $n$, $\sL$ be a line bundle on $X$ and let $\phi: \E \rightarrow \E^\star \otimes \sL$ be a generic symmetric morphism.
If $S^2 \E^\star \otimes \sL$ is globally generated, then $D_k(\phi)$ is empty or of codimension $\binom{n-k+1}{2}$. Moreover, $\sing(D_k(\phi))$ is contained in $D_{k-1}(\phi)$.
\end{lemma}
\begin{proof}
The exact sequence $H^0(S^2 \E^\star \otimes \sL) \otimes \sO_X \rightarrow S^2 \E^\star \otimes \sL \rightarrow 0$
induces a projection that is everywhere of maximal rank
$$X \times H^0(S^2 \E^\star \otimes \sL) \xrightarrow{p} S^2 \E^\star \otimes \sL$$
One can define the variety $\Sigma_k$ inside the total space of $S^2 \E^\star \otimes \sL$ that is composed on each fibre of the $n \times n$ matrices with complex coefficients with rank less or equal than $k$. $\Sigma_k$ has codimension equal to $\binom{n-k+1}{2}$.
Now consider the following diagram where $Z$ is the preimage of $\Sigma_k$ by $p$. 
$$\xymatrix{
& Z \ar[r]\ar[d]\ar[dl]_q & \Sigma_k\ar@{^{(}->}[d]\\
 H^0(S^2 \E \otimes L) & X \times H^0(S^2 \E \otimes L) \ar[r]^-p\ar[l] & S^2 \E \otimes L
}$$
$Z$ is composed of the pairs $(x,\phi)$ where $x$ is a point of $X$, $\phi: \E \rightarrow \E^\star \otimes \sL$ is a symmetric map and the rank of the induced map $\phi_x$ is less or equal than $k-1$. Then, the preimage by $q$ of a generic element $\phi_0$ in $H^0(S^2 \E \otimes L)$ coincides with $\{(x,\phi_0) \st \rk(\phi_0|_x) \leq k-1\}$, which is exactly the definition of $D_k(\phi_0)$.
Notice that $\sing(Z) = p^{-1} \sing(\Sigma_k)$ thus, by restricting $p$, one obtains
$$Z \smallsetminus \sing(Z) \xrightarrow{p|_{Z \smallsetminus \sing(Z)}} H^0(S^2 \E \otimes L).$$
If the image of $p|_{Z \smallsetminus \sing(Z)}$ is dense, then $D_k(\phi)$ is smooth by the generic smoothness theorem, and so $\sing(D_k(\phi)) \subset D_{k-1}(\phi)$. If the image is not dense, then $D_k(\phi)$ is empty for generic $\phi$.
\end{proof}

The following proposition ensures that for any cubic fourfold $Y$ with an associated nodal sextic curve $C$ is it possible to find a smooth cubic fivefold $\bar{Y}$ with the associated sextic surface $\bar{C}$ having singularities of codimension $3$ in $\P^3$.

\begin{prop}\label{prop:LocusCodThree}
Let $\F$ be the vector bundle $S^2 \E \otimes \sO_{\P^3}(1)$ on $\P^3$, where $\E$ is $3\sO_{\P^3} \oplus \sO_{\P^3}(-1)$. Let $\T$ be the linear system of sections of $\F$ defined by the choice of the second matrix in (\ref{Eqn:matrixform}). Then a generic section of the linear system $\T$ defines a smooth cubic fivefold and the degeneracy locus $D_2$ of the morphism corresponding to this section is of codimension $3$ in $\P^3$. Moreover, $D_2$ is composed of a finite number of ordinary double points and there is a bijection between $D_2 \cap \{x_3 = 0\}$ and the nodes of the sextic curve $C$.
\end{prop}
\begin{proof}
The generic section of $\T$ defines a smooth fivefold. Notice that the expected codimension of the $2-$degeneracy locus of a section of $\F$ is $\binom{4-2+1}{2}=3$. 
Let $x$ be a point in $\P^3 \smallsetminus \{x_3 = 0\}$. The stalk $\F_x$ is globally generated by the sections in $\T$ because $\bar{q}_k$ contains the monomial $\alpha_k x_3$ and $f$ contains the monomial $\beta x_3^2$, $\alpha_k$ and $\beta$ in $\C$. The matrix of complex numbers obtained after evaluating the second matrix in (\ref{Eqn:matrixform}) in the stalk $\F_x$ is generic, since $\bar{l}_{ij}, \alpha_k$ and $\beta$ are arbitrarily chosen. This proves that $S^2 \F \otimes \sO(3)$ is globally generated by sections of $\T$ where $x_3\neq 0$.
Then, it is possible to apply Lemma \ref{Lemma:SymmetricCodimension} to the restriction of the bundles to $\P^3 \smallsetminus \{x_3 = 0\}$.
It remains to prove that $D_2 \cap \{x_3=0\}$ has codimension three, but this is a direct consequence of the setting of the problem, in which the sextic curve, which coincides with $D_2 \cap \{x_3=0\}$, has at most singularities of codimension $2$ in $\P^2$.

Since $D_2$ is of dimension $0$ in $\P^3$ then it must be finite. Let $x$ be a point in $D_2$. Proving that $x$ is an ordinary double point of $\bar{C}$ is the same as proving that the fibre over $x$ is composed of $A$ and the singular quadric given by two intersecting planes. Consider the hyperplane section of $\P^6$ given by a hyperplane $H$ containing $A$ and $x$. This gives rise to a cubic fourfold $Y_H$ containing $A$ with a projection to the plane $B_H := \bar{B} \cap H$ still containing the point $x$. By construction the fibre over $x$ is the same, and then the result follows by Proposition $1.2$ of \cite{BPrym}. The same proposition ensures that also the curve $C$ has only nodes. In both cases, each one is determined by the fact that the quadric fibration has a fibre that consists of the union of two planes, and this is sufficient to conclude the proof of the last part of the proposition.
\end{proof}

\begin{cor} \label{cor:pullbackb0}
At the level of noncommutative varieties one has
$$(\P^2,\B_0) = (\P^3,\bar{\B_0}) \times_{\bar{B}} B,$$
where $\bar{\B_0}$ is the even Clifford algebra of the quadratic form over $\P^3$.
\end{cor}
\begin{proof}
By Proposition \ref{prop:LocusCodThree}, the quadric fibration for the cubic fourfold $Y$ is obtained from the one of the cubic fivefold $\bar{Y}$ by base change with respect to the embedding $B \xrightarrow{i} \bar{B}$, where $B$ is the plane $\{x_3=0\}$ in $\P^3$. It follows that $\B_0$ is the pullback along $i$ of $\bar{\B_0}$, and this concludes the proof.
\end{proof}

Let $Y$ be a cubic fourfold containing a plane whose associated sextic curve $C$ is nodal. Let $S$ be the (singular) double cover of $B$ ramified along $C$ and $S^+$ be the minimal resolution of the singularities of $S$. Let $\B_0$ be the Clifford algebra on $B$ of the quadric fibration associated to $Y$.

\begin{prop} \label{prop:final}
There exists an Azumaya algebra $\A$ defined on $S^+$ such that the pushforward to $B$ is isomorphic to $\B_0$.
\end{prop}

\begin{proof}
Proposition \ref{prop:LocusCodThree} guarantees the existence of a cubic fivefold $\bar{Y} \subset \P^6$ such that $Y$ is contained in $\bar{Y}$ and the degeneracy locus of the associated quadric fibration is a sextic surface $\bar{C}$ in $\P^3$ with at most ordinary double points. Taking a specific hyperplane section of $\P^6$ gives back the original picture of $Y \subset \P^5$. The quadric fibration on $\P^3$ arising from $\bar{Y}$ satisfies the hypothesis of \cite[Theorem 1.1]{K2}. In particular the construction of sections $4$ and $5$ of \cite{K2} can be carried out to obtain the following diagram
\begin{equation} \label{Eqn:NoPullback}
	\xymatrix{
 S^+\ar[d] & X^+\ar[d] \\
 S\ar[d] & X\ar[d] \\
 \P^2\ar@{^{(}->}[r]  & \P^3 
}
\end{equation}
Here $S$ and $X$ are respectively the double covering of $\P^2$ ramified over the sextic curve $C$ and of $\P^3$ ramified over the sextic surface $\bar{C}$ induced from the centres of the corresponding Clifford algebras. The spaces $S^+$ and $X^+$ are resolutions of the singularities of $S$ and $X$. The structure of the exceptional locus over a singular point of $X$ is described in \cite[Proposition 4.4]{K2} and is isomorphic to $\P^1$. Notice that the projective space $\P^3$ coincides with the smooth base $Y$ in the notation of \cite{K2}.
Lemma \ref{Lemma:FiberProductBranch} and Lemma \ref{Lemma:FiberProductResolution} located below, ensure that the two vertical maps of the diagram (\ref{Eqn:NoPullback}) are actually fibre products, hence the diagram can be completed as follows
\begin{equation}  \label{Eqn:Pullback}
	\xymatrix{
& S^+\ar[d]^{\tau^+}\ar@{^{(}->}[r]^g \pullbackcorner &  X^+\ar[d]^{\sigma^+}\ar@{.}[r] & \B^+\\ 
\tilde{\B_0}\ar@{.}[r] & S\ar[d]^{\pi}\ar@{^{(}->}[r] \pullbackcorner & X\ar[d]^f &\\
 \B_0\ar@{.}[r] & \P^2\ar@{^{(}->}[r]  & \P^3\ar@{.}[r] & \bar{\B_0}
}
\end{equation}
where $\B_0$ and $\bar{\B_0}$ are the Clifford algebras associated to the quadric fibrations related to $Y$ and $\bar{Y}$ described in the preliminaries and the dotted arrows denote the fact that the algebras are defined over the corresponding schemes. 
Kuznetsov proved that there is a $\P^1$-bundle $M^+ \rightarrow X^+$, where $M^+$ is a flip of the Fano scheme $M$ of lines over the family of quadrics associated to $X$. That is sufficient to obtain an Azumaya algebra $\B^+$ on $X^+$ as done in \cite[Proposition 5.5]{K2}.  By Corollary \ref{cor:pullbackb0} one has $\B_0 \cong i^* \bar{\B_0}$. Moreover, \cite[Lemma 5.7]{K2} ensures that the push forward of $\B^+$ along $(f \circ \sigma^+)$ is isomorphic to $\bar{\B_0}$. 

The two squares in the diagram (\ref{Eqn:Pullback}) can be composed to obtain the following
\begin{equation} \label{Dgr:comm}
	\xymatrix{
 g^*(\B^+)\ar@{.}[r] & S^+\ar@{^{(}->}[r]^{g}\ar[d]^k \pullbackcorner & X^+\ar@{.}[r]\ar[d]^{h} & \B^+\\
 \B_0\ar@{.}[r] & \P^2\ar@{^{(}->}[r]^{i}  & \P^3\ar@{.}[r] & \bar{\B_0}
}
\end{equation}
where $k := \pi \circ \tau^+$ and $h := f \circ \sigma^+$.

The inclusion $i$ is strict, hence Proposition \ref{prop:pullbackazumvar} ensures that square (\ref{Dgr:comm}) gives a base change diagram of noncommutative varieties. Moreover, the pullback $\A := g^*(\B^+)$ is an Azumaya algebra on $S^+$.

It remains to prove that the pushforward $(\pi \circ \tau^+)_*(\A)$ is isomorphic to $\B_0$. Recalling that $\B_0 \cong i^* \bar{\B_0}$, $\bar{\B_0} \cong h_* \B^+$ and $\A := g^* \B^+$, one obtains the following base change problem:
prove that $k_* g^* \B^+$ is equivalent to $i^* h_* \B^+$. This depends on the Square (\ref{Dgr:comm}) being exact cartesian, see Definition \ref{Defn:ExactCartesian}. We can apply Lemma \ref{Lemma:ExactCartesian}, since the inclusion $i$ is strict and closed and the codimensions of $S^+$ and $\P^2$ on $X^+$ are the same, and hence the square (\ref{Dgr:comm}) turns out to be exact cartesian giving that the pushforward $(\pi \circ \tau^+)_*(\A)$ is isomorphic to $\B_0$. which concludes the proof.
\end{proof}

\noindent Since being Azumaya is a local property, the pushforward of $\A$ to $S$, restricted to the complement of $\sing(C)$, is isomorphic to $\tilde{\B}_0$, the Azumaya algebra on $S \smallsetminus \sing(C)$ provided by Proposition \ref{Prop:AzumayaAlgebraEvenFibration}. The following lemmata, used in the proof of Proposition \ref{prop:final}, hold under the same assumptions.

\begin{lemma} \label{Lemma:FiberProductBranch}
The fibre product along an inclusion of a branched covering is again a covering, ramified on the restriction of the ramification locus.
\end{lemma}
\begin{proof}
One has to prove that, if $V \subset W$ and $\tilde{W}$ is the covering of $W$ ramified over $C$, then the fibre product
\begin{equation*}
	\xymatrix{
V \times_W \tilde{W} \ar[r]\ar[d] & \tilde{W}\ar[d]^f\\
V\ar@{^{(}->}[r]^i & W
}
\end{equation*}
is isomorphic to the covering $\tilde{V} \xrightarrow{\pi} V$ of $V$ ramified over $i^{-1}(C)$. This can be proved by using the definition of the fibre product. If $p \in \tilde{V}$ one can consider $f^{-1}(i \circ \pi)(p)$ and this gives a map to $\tilde{W}$. Moreover, if $Z \xrightarrow{a} \tilde{W}$ and  $Z \xrightarrow{b} V$ are such that $f \circ a = i \circ b$, then one obtains a map $Z \rightarrow \tilde{V}$ that makes the diagram commute, giving the isomorphism between $\tilde{V}$ and $V \times_W \tilde{W}$.
\end{proof}

\begin{lemma} \label{Lemma:FiberProductResolution}
The fibre product of $X^+ \xrightarrow{\sigma^+} X$ along the inclusion $S \rightarrow X$ is the minimal resolution of singularities of $S$.
\end{lemma}
\begin{proof}
By construction, the inclusion sends singular points of $S$ to singular points of $X$. By the description of the exceptional locus of $X^+$, the pullback map to $S$ is an isomorphism on the smooth points of $S$ and is a contraction of a $\P^1$ on the singular points. The singularities of $X$, in particular the ones in common with $S$, are cones on singular quadrics degenerated to two planes. A small resolution $X^+$, which exists by the work of Kuznetsov, is hence obtained by blowing up one of these planes for each singular point. Since the sextic curve was obtained by taking an hyperplane section of the sextic surface, it follows that $S$ is a Cartier divisor inside $X$. It remains to check that the $\P^1$ contracted by the resolution is not all contained in the hyperplane $\{x_3=0\}$. That is implied by the last part of Proposition \ref{prop:LocusCodThree}, since a node on $X$ remains a node on $S$, hence the fibre product $X^+ \times_X S$ is smooth. This implies that the restriction of the resolution of singularities on $X$ coincides with the resolution of the singularities on $S$.
\end{proof}

\noindent It is now possible to give a proof for Theorem \ref{teor:extension}. The notation is the same used in this whole section. Notice that it is also possible to prove directly the equivalence $\T_Y \cong \der^b(S^+,\A)$ by using the result of Proposition \ref{prop:final} and then retracing the proof of \cite[Theorem 1.1]{K2}.

\begin{proof}[Proof of Theorem \ref{teor:extension}]
The theorem can be proved by a sequence of equivalences; recall first the one given by Proposition \ref{prop:stepone}:
\begin{equation} \label{EStep1}
\T_Y \cong \der^b(B,\B_0)\text{.}
\end{equation}
Corollary \ref{cor:pullbackb0} gives the following isomorphism of noncommutative varieties:
$$(\P^2,\B_0) = (\P^3,\bar{\B_0}) \times_{\bar{B}} B\text{.}$$
As observed in the proof of \cite[Lemma 3.2]{QUADRIC}, the base change of such algebras can be described in terms of quadric fibrations. At the level of derived categories this gives
\begin{equation} \label{EStep2}
\der^b(B,\B_0) \cong \der^b(B,i^*\bar{\B_0}) \cong \der^b\left((\bar{B},\bar{\B_0}) \times_{\bar{B}} B\right),
\end{equation}
that is the base change of $\der^b(\bar{B},\bar{\B_0})$ to $\P^2$. 

Then, as in the proof of Theorem 1.1 in \cite{K2}, there is an equivalence between the derived category of $\bar{B}$ twisted by $\bar{\B_0}$ and the derived category of $X^+$ twisted by the Azumaya algebra $\B^+$.
\begin{equation} \label{EStep3}
\der^b(\bar{B},\bar{\B_0}) \cong \der^b(X^+,\B^+)\text{.}
\end{equation}

Finally, as in the proof of Proposition \ref{prop:final}, since the map $i$ in the square (\ref{Dgr:comm}) is strict, it is possible to apply Proposition \ref{prop:pullbackazumvar}, obtaining
$$(S^+,\A) = (X^+,\B^+) \times_{\bar{B}} B\text{.}$$
In the case of Azumaya algebras, the base change can be described in terms of $\P^1$-fibrations, and for them, a semiorthogonal decomposition is provided in \cite{BER}. Hence, as for the equivalence (\ref{EStep2}),
\begin{equation} \label{EStep4}
\der^b(S^+,\A) \cong \der^b(S^+,g^*\B^+) \cong \der^b\left((X^+,\B^+) \times_{\bar{B}} B\right),
\end{equation}

The theorem is proved by combining, in order, the equivalences (\ref{EStep1}), (\ref{EStep2}), (\ref{EStep3}) and (\ref{EStep4}).
\end{proof}


\medskip

{\small\noindent{\bf Acknowledgements.}
During the preparation of this paper, the author was supported by the Department of Mathematics and Natural Sciences of University of Stavanger in the framework of the grant 230986 of the Research Council of Norway.
It is a pleasure to thank Michal Kapustka, Paolo Stellari and Tom Sutherland for their useful suggestions and helpful discussions. The comments of an anonymous referee helped a lot in correcting many mistakes in the first version of this paper. The author is especially grateful to Alexander Kuznetsov for his valuable comments on a preliminary version of this paper.
}

\end{document}